\documentclass[10pt]{article}
\usepackage{latexsym,amsfonts,amsmath,amsthm,amssymb, epsfig}

\newtheorem{thm}{Theorem}

\newtheorem{cor}[thm]{Corollary}
\newtheorem{con}{Conjecture}

\theoremstyle{remark}
\newtheorem{rem}{Remark}

\theoremstyle{definition}

\newcommand{\R}{\mathbb{R}}

\newcommand{\C}{\mathbb{C}}
\newcommand{\Z}{\mathbb{Z}}
\newcommand{\E}{\mathbb{E}}
\newcommand{\T}{\mathbb{T}}

\newcommand{\diag}{{\rm diag\ }}
\newcommand{\Cov}{{\rm Cov}}

\newcommand\INT{{\rm INT}}

\date{}

\title{A variant of Schur's product theorem\\ and its applications}
\author{Jan Vyb\'iral\footnote{Dept. of Mathematics FNSPE, Czech Technical University in Prague, Trojanova 13, 12000 Prague, Czech Republic;
the research of this author was supported by the grant P201/18/00580S of the Grant Agency of the Czech Republic
and from European Regional Development Fund-Project ``Center for Advanced Applied Science'' (No. CZ.02.1.01/0.0/0.0/16\_019/0000778),
\texttt{jan.vybiral@fjfi.cvut.cz}}}

\begin{document}
\maketitle
\begin{abstract}
We show the following version of the Schur's product theorem.
If $M=(M_{j,k})_{j,k=1}^n\in\R^{n\times n}$ is a positive semidefinite matrix with all entries on the diagonal equal to one, then the matrix
$N=(N_{j,k})_{j,k=1}^n$ with the entries $N_{j,k}=M_{j,k}^2-\frac{1}{n}$ is positive semidefinite.
As a corollary of this result, we prove the conjecture of E. Novak on intractability of numerical integration on the space of trigonometric polynomials
of degree at most one in each variable. Finally, we discuss also some consequences for Bochner's theorem, covariance matrices of $\chi^2$-variables, and
mean absolute values of trigonometric polynomials.
\end{abstract}

\section{Introduction}\label{Sec:1}

Over twenty years ago, motivated by tractability studies of numerical integration, E. Novak \cite{Erich} made the following conjecture:
\begin{con}[E. Novak]\label{con:Erich}
The matrix
$$
\Bigl\{\prod_{i=1}^d\frac{1+\cos(x_{j,i}-x_{k,i})}{2}-\frac{1}{n}\Bigr\}_{j,k=1}^n
$$
is positive semidefinite for all $n,d\ge 2$ and all choices of $x_1,\dots,x_n\in\R^d.$
\end{con}
Erich Novak published this conjecture also in NA Digest in November 1997 and tested it numerically.
It also appeared as Open Problem 3 in \cite{NW08}. Nevertheless, it seems that up to now the problem remained unsolved.

Further extensive numerical tests were provided by A. Hinrichs and the author in \cite{HV},
all supporting the belief that Conjecture \ref{con:Erich} is true.
The main difficulty in actually proving this conjecture turned out to be to identify the relevant properties of the function $f(t)=(1+\cos t)/2$,
which play a role in this question.
Led by other numerical tests, \cite{HV} conjectures that the same property is true for all positive positive-definite functions
with the value at zero equal to one. Unfortunately, even in this form, the conjecture remained unsolved.

Our proof of this conjecture, which we present in this note, is based on a certain simple but rather unexpected and apparently unknown property of the Hadamard product.
To state it we need a few simple notations.
We say that $M\in\C^{n\times n}$ is positive semidefinite if $c^*Mc\ge 0$ for all $c\in\C^n.$
Note that by \cite[Theorem 4.1.4.]{HornJohnson} this means that $M$ is necessarily self-adjoint.
By \cite[Theorems 4.1.8. and 4.1.10.]{HornJohnson}, $M\in\R^{n\times n}$ is positive semidefinite if it is symmetric and $c^TMc\ge 0$ for all $c\in\R^n.$
If $M,N\in\C^{n\times n}$, we denote by $M\circ N$ their Hadamard product \cite{Horn}, i.e.,
a matrix with entries $(M\circ N)_{j,k}=M_{j,k}\cdot N_{j,k}$ for all $j,k=1,\dots,n.$ Furthermore, the partial
ordering $M\succeq N$ of self-adjoint matrices means that $M-N$ is positive semidefinite.
We denote by $\diag M=(M_{1,1},\dots,M_{n,n})^T$ the vector of diagonal entries of $M$ whenever $M\in \C^{n\times n}$
and by $\overline{M}$ the matrix with entries $(\overline{M})_{j,k}=\overline{M_{j,k}}$.
Finally, $e^n=(1,\dots,1)^T\in\R^n$ stands for the vector of ones and $E_n=e^n(e^n)^T\in\R^{n\times n}$ is a matrix with all entries equal to one.

Using this notation, the main result of this paper then reads as follows.
\begin{thm}\label{thm:main1} Let $M\in\C^{n\times n}$ be a self-adjoint positive semidefinite matrix. Then
$$
M\circ \overline M\succeq \frac{1}{n} (\diag M)(\diag M)^T.
$$
\end{thm}
Conjecture \ref{con:Erich} will be resolved by a special case of Theorem \ref{thm:main1}, which we formulate as a corollary.
By restricting to real positive semidefinite matrices with ones on the diagonal, it gives a uniform lower bound for the class of real correlation matrices.
\begin{cor}\label{cor:2} Let $M\in\R^{n\times n}$ be symmetric positive semidefinite with $M_{j,j}=1$ for all $j=1,\dots,n$. Then
$$
M\circ M\succeq \frac{1}{n} E_n.
$$
\end{cor}

We prove Theorem \ref{thm:main1} in Section \ref{Sec:2}.
Let us note that it actually resembles the original work of Schur \cite{Schur}.
This section also includes 
several generalizations of Theorem \ref{thm:main1}.
Sections \ref{Sec:3} and \ref{Sec:4} discuss the connections
with Bochner's theorem and with covariance matrices of multivariate $\chi^2$ random variables. Finally, Section \ref{Sec:5} gives an account on the
numerical integration, which was the original motivation of E. Novak, and shows how Corollary \ref{cor:2} implies Conjecture \ref{con:Erich}.

\section{Proof of Theorem \ref{thm:main1} and its variants}\label{Sec:2}

We begin with the following theorem, which leads us to a proof of Theorem \ref{thm:main1}.

\begin{thm}\label{thm:ref1} Let $M\in\C^{n\times n}$ be positive semidefinite. Let $M=AA^*$ and let $w=Ae$ be the vector of row sums of $A$. Then
$$
M\succeq \frac{1}{n} ww^*.
$$
\end{thm}
\begin{proof}
We need to show that
\begin{equation}\label{eq:PD1}
c^*Mc\ge \frac{1}{n}c^*ww^*c
\end{equation}
for every $c\in\C^n$. We start with a reformulation of the left-hand side of \eqref{eq:PD1}
and compute for arbitrary $c\in\C^n$
\begin{align}\label{eq:PD2}
c^*Mc&=c^*AA^*c=\|A^*c\|_2^2=\sum_{k=1}^n |(A^*c)_k|^2
=\sum_{k=1}^n\Bigl|\sum_{j=1}^n\overline{A_{j,k}}c_j\Bigr|^2.
\end{align}
Similarly, we reformulate the right-hand side of \eqref{eq:PD1}
\begin{equation}\label{eq:PDX}
\frac{1}{n}c^*ww^*c=\frac{1}{n}\sum_{j=1}^n \overline{c_j}w_j\sum_{k=1}^n c_k\overline{w_k}=\frac{1}{n}\Bigl|\sum_{j=1}^nc_j\overline{w_j}\Bigr|^2.
\end{equation}
Furthermore, by using triangle inequality and the (sharp) inequality
between $\ell_2^n$ and $\ell_1^n$ norm \cite[Problem 5.4.P3]{HornJohnson} (or just the Cauchy-Schwarz inequality), we obtain
\begin{align}
\notag\Bigl|\sum_{j=1}^n c_j \overline{w_{j}}\Bigr|&=\Bigl| \sum_{j=1}^nc_j\sum_{k=1}^n \overline{A_{j,k}}\Bigr|
\le \sum_{k=1}^n \Bigl|\sum_{j=1}^nc_j \overline{A_{j,k}}\Bigr|\\
\label{eq:PD3}&\le \sqrt{n} \Bigl(\sum_{k=1}^n \Bigl|\sum_{j=1}^nc_j\overline{A_{j,k}}\Bigr|^2\Bigr)^{1/2}.
\end{align}
Taking the square of \eqref{eq:PD3} and combining it with \eqref{eq:PD2} and \eqref{eq:PDX} finishes the proof.
\end{proof}

\begin{thm}\label{thm:ref2} Let $M,N\in\C^{n\times n}$ be positive semidefinite with $M=AA^*$ and $N=BB^*.$
Let $w=(A\circ B)e$ be the vector of row sums of $A\circ B$. Then
$$
M\circ N\succeq (A\circ B)(A\circ B)^*\succeq\frac{1}{n}ww^*.
$$
\end{thm}
\begin{proof} We denote by $A^j$ and $B^j$ the columns of $A$ and $B$, respectively. Then
\begin{align*}
M\circ N&=(AA^*)\circ (BB^*)=\Bigl(\sum_{j=1}^n A^j (A^j)^*\Bigr)\circ\Bigl(\sum_{k=1}^n B^k (B^k)^*\Bigr)\\
&=\sum_{j,k=1}^n (A^j (A^j)^*)\circ(B^k (B^k)^*)=\sum_{j,k=1}^n (A^j\circ B^k)(A^j\circ B^k)^*\\
&\succeq \sum_{j=1}^n (A^j\circ B^j)(A^j\circ B^j)^*=(A\circ B)(A\circ B)^*
\end{align*}
and the theorem follows by applying Theorem \ref{thm:ref1} to the matrix $(A\circ B)(A\circ B)^*$.
\end{proof}
Theorem \ref{thm:main1} now follows from Theorem \ref{thm:ref2} if we choose $N=\overline{M}$ and $B=\overline{A}.$ Indeed, the row sums of $A\circ \overline{A}$
are the diagonal entries of $M$. Furthermore, choosing $N=M$ and $B=A$ in Theorem \ref{thm:ref2} leads to
$$
M\circ M\succeq \frac{1}{n}ww^*,
$$
where $w=(A\circ A)e$ is the vector of row sums of $A\circ A$.

\section{Bochner's theorem}\label{Sec:3}

If $\mu$ is a finite Borel measure on $\R^d$, then its Fourier transform is given by
$$
g(\xi)=\hat \mu(\xi)=\int_{\R^d} e^{-2\pi i \xi\cdot x}d\mu(x),\quad\xi\in\R^d,
$$
where $\xi\cdot x=\langle \xi,x\rangle$ is the inner product of $x\in\R^d$ and $\xi\in\R^d.$
The classical Bochner theorem (actually its easy part), see \cite{Bochner1, Bochner2}, states that the matrix
$$
(g(\xi_j-\xi_k))_{j,k=1}^n
$$
is positive semidefinite for every choice of $\xi_1,\dots,\xi_n\in\R^d$.

The proof follows by a simple calculation, as we have for every $c\in\C^d$
\begin{align}
\notag\sum_{j,k=1}^n c_j\overline{c_k} g(\xi_j-\xi_k)&=\sum_{j,k=1}^n c_j\overline{c_k} \int_{\R^d} e^{-2\pi i (\xi_j-\xi_k)\cdot x}d\mu(x)\\
\label{eq:Bochner1}&=\int_{\R^d} \sum_{j,k=1}^n c_j\overline{c_k}  e^{-2\pi i \xi_j\cdot x} e^{2\pi i \xi_k\cdot x}d\mu(x)\\
\notag&=\int_{\R^d} \Bigl|\sum_{j=1}^n c_j  e^{-2\pi i \xi_j\cdot x}\Bigr|^2d\mu(x)\ge 0.
\end{align}
We refer to \cite{Stewart} for a classical overview of positive definite functions.
Theorem \ref{thm:main1} leads to the following modification of one part of Bochner's Theorem.

\begin{thm}\label{thm:Bochner1} Let $\mu$ be a finite Borel measure on $\R^d$ and let $g$ be its Fourier transform. Then
$$
(|g(\xi_j-\xi_k)|^2)_{j,k=1}^n\succeq \frac{g^2(0)}{n}\cdot E_n
$$
for every choice of $\xi_1,\dots,\xi_n\in\R^d.$
\end{thm}
\begin{proof}
By the easy part of Bochner's theorem, see \eqref{eq:Bochner1}, we know that the matrix $M=(M_{j,k})_{j,k=1}^n$ with $M_{j,k}=g(\xi_j-\xi_k)$
is positive semidefinite. Furthermore, $M_{j,j}=g(0)=\int_{\R^d}1d\mu(x)$ is real for every $j=1,\dots,n$.
The result then follows by invoking Theorem \ref{thm:main1}.
\end{proof}

Theorem \ref{thm:Bochner1} can be generalized to other locally compact abelian groups by just applying the corresponding version of Bochner's theorem.
In this way, one can prove for example the following version for the torus $\T$; see Conjecture 3 in \cite{HV}.
\begin{thm} Let $\alpha=(\alpha_j)_{j\in\Z}$ be a non-negative summable sequence, i.e., $\alpha_j\ge 0$ for every $j\in\Z$ and $\alpha\in\ell_1(\Z)$.
Let $g(x)=\sum_{j\in\Z}\alpha_j e^{ijx}$ for every $x\in\T$. Then
$$
(|g(\xi_j-\xi_k)|^2)_{j,k=1}^n\succeq \frac{g^2(0)}{n}\cdot E_n
$$
for every choice of $\xi_1,\dots,\xi_n\in\T.$
\end{thm}

Theorem \ref{thm:Bochner1} has an interesting reformulation in the language of independent random variables.
\begin{thm}\label{thm:Bochner:2} Let $\mu$ be a Borel probability measure on $\R^d$.
Let $\omega_1$ and $\omega_2$ be two independent random vectors, both distributed with respect to $\mu$.
Then
\begin{equation}\label{eq:Bochner:2}
\E \Bigl|\sum_{j=1}^n c_j e^{-2\pi ix_j\cdot(\omega_1-\omega_2)}\Bigr|^2\ge \frac{1}{n}\Bigl|\sum_{j=1}^n c_j\Bigr|^2
\end{equation}
for every choice of $c\in\C^n$ and $x_1,\dots,x_n\in\R^d.$
\end{thm}
\begin{proof}
We denote again by $g$ the Fourier transform of $\mu$.
The proof follows from Theorem \ref{thm:Bochner1} and the following direct calculation
\begin{align*}
\E \Bigl|\sum_{j=1}^n c_j e^{-2\pi ix_j\cdot(\omega_1-\omega_2)}\Bigr|^2
&=\E \sum_{j,k=1}^n c_j \overline{c_k}e^{-2\pi i(x_j-x_k)\cdot(\omega_1-\omega_2)}\\
&=\sum_{j,k=1}^n c_j \overline{c_k}\,[\E\, e^{-2\pi i(x_j-x_k)\cdot \omega_1}]\cdot [\E\, e^{2\pi i(x_j-x_k)\cdot \omega_2}]\\
&=\sum_{j,k=1}^n c_j\overline{c_k} |g(x_j-x_k)|^2.
\end{align*}
\end{proof}

\begin{rem}
If $d=1$ and $\omega_1,\omega_2$ are independent and identically distributed standard normal random variables, then $\omega_1-\omega_2$ is again a normal random variable.
After rescaling, \eqref{eq:Bochner:2} gives for every $y_1,\dots,y_n\in\R$ and every $c\in\C^n$
\begin{equation}\label{eq:square:3}
{\mathbb E}\Bigl|\sum_{j=1}^n c_j e^{iy_j \omega}\Bigr|^2\ge \frac{1}{n}\Bigl|\sum_{j=1}^n c_j\Bigr|^2.
\end{equation}
\end{rem}

\section{Covariance matrices}\label{Sec:4}

Corollary \ref{cor:2} has an interesting consequence for covariance matrices of multivariate $\chi^2$-distributions (with one degree of freedom).
Let $M\in\R^{n\times n}$ be a positive semidefinite matrix with ones on the diagonal
and let $X_1,\dots,X_n$ be Gaussian random variables with zero mean, unit variance and $\Cov(X_j,X_k)=M_{j,k}$ for all $j,k=1,\dots,n.$
Using Wick's theorem, see \cite{Iss1, Iss2, Wick},
we observe that
$$
\Cov(X_j^2,X_k^2)=\E[X_j^2X_k^2]-\E[X_j^2]\cdot \E[X_k^2]=2(\E[X_jX_k])^2=2M_{j,k}^2.
$$
Applying 
Corollary \ref{cor:2} we obtain
$$
(\Cov(X_j^2,X_k^2))_{j,k=1}^n=(2M_{j,k}^2)_{j,k=1}^n\succeq \frac{2}{n}\cdot E_n.
$$
Thus we have proven 
\begin{thm}\label{thm:chi1}
Let $X=(X_1,\dots,X_n)$ be a vector of standard normal random variables. Then
$$
(\Cov(X_j^2,X_k^2))_{j,k=1}^n\succeq \frac{2}{n}\cdot E_n.
$$
\end{thm}

\section{Numerical integration}\label{Sec:5}

In this section, we summarize the approach of \cite{Erich}, where E. Novak studied how well the quadrature formulas
\begin{equation}\label{eq:1}
Q_n(f)=\sum_{i=1}^n c_if(x_i),\qquad c_i\in\R,\quad x_i\in[0,1]^d
\end{equation}
approximate the integral $\INT_d(f)=\int_{[0,1]^d}f(x)dx.$
Here $f$ belongs to a unit ball of a Hilbert space $F_d$, that is a $d$-fold tensor product of a space $F_1$, which is a three-dimensional Hilbert space
with an orthonormal basis given by the functions
$$
e_{1}(x)=1,\quad e_{2}(x)=\cos(2\pi x),\quad e_{3}(x)=\sin(2\pi x),\quad x\in[0,1].
$$
Hence $F_d$ is a $3^d$-dimensional Hilbert space.
The point evaluation $\delta_x:f\to f(x)$ may be written in the form
$$
f(x)=\langle f,\delta_x\rangle_{F_d}\quad \text{with}\quad
\delta_x(z)=\prod_{j=1}^d[1+\cos(2\pi (x_j-z_j))].
$$
In this way, $F_d$ becomes a reproducing kernel Hilbert space with the kernel
$$
K_d(x,y)=\langle\delta_x,\delta_y\rangle_{F_d}=\prod_{j=1}^d[1+\cos(2\pi(x_j-y_j))],\quad x,y\in[0,1]^d.
$$
This allows us to compute the worst-case error of $Q_n$ given by \eqref{eq:1} as
\begin{align*}
e^{\rm wor}(Q_n)^2&=\sup_{||f||_{F_d}\le 1}|\INT_d f-Q_n(f)|^2
=\biggl\|1-\sum_{j=1}^n c_j \delta_{x_j}\biggr\|^2_{F_d}\\
&=1-2\sum_{j=1}^n c_j+\sum_{j,k=1}^nc_jc_k K_d(x_j,x_k).
\end{align*}
If all $c_j$'s are positive, we can use the positivity of $K_d$ and obtain
$$
e^{\rm wor}(Q_n)^2\ge 1-2\sum_{j=1}^nc_j+\sum_{j=1}^n c_j^22^d.
$$
For the optimal choice $c_j=2^{-d}$ this becomes
\begin{equation}\label{eq:2}
e^{\rm wor}(Q_n)^2\ge \max(1-n2^{-d},0).
\end{equation}
This estimate shows the intractability of numerical integration on $F_d$ with quadrature formulas with positive weights since
for a fixed error the number $n$ of sample points needs to grow exponentially with the dimension $d$.

To estimate $e^{\rm wor}(Q_n)^2$ from below for quadrature rules with general weights $c$, we use the fact
that the projection of any $y\in F_d$ onto the ray generated by $x\in F_d$ is given by
$\displaystyle \frac{\langle y,x\rangle_{F_d} x}{\langle x,x\rangle_{F_d}}$ and its norm is equal to $\displaystyle \frac{|\langle y,x\rangle_{F_d}|}{\|x\|_{F_d}}$.
In this way we obtain
\begin{align}
\inf_{c_j,x_j} \biggl\|1-\sum_{j=1}^n c_j \delta_{x_j}\biggr\|^2_{F_d}
\notag&=\inf_{c_j,x_j}\inf_{\alpha\in \R} \biggl\|1-\alpha\sum_{j=1}^n c_j \delta_{x_j}\biggr\|^2_{F_d}\\
\label{eq:3}&=\inf_{c_j,x_j}\left\{\|1\|_{F_d}^2- \frac{\displaystyle\Bigl|\Bigl\langle 1,\sum_{j=1}^n c_j\delta_{x_j}\Bigr\rangle_{F_d}\Bigr|^2}
{\displaystyle\Bigl|\Bigl\langle \sum_{j=1}^n c_j\delta_{x_j},\sum_{j=1}^n c_j\delta_{x_j}\Bigr\rangle_{F_d}\Bigl|}\right\}\\
\notag&=1-\sup_{c_j,x_j}\frac{\displaystyle \biggl(\sum_{j=1}^nc_j\biggr)^2}
{\displaystyle \sum_{j,k=1}^n c_jc_k K_d(x_j,x_k)}.
\end{align}

Erich Novak conjectured that the estimate \eqref{eq:2} applies also for
quadrature formulas \eqref{eq:1} with general weights, which by \eqref{eq:3} is equivalent to Conjecture \ref{con:Erich}.

Finally, let us show how 
Corollary \ref{cor:2} implies that Conjecture \ref{con:Erich} is correct.
We define matrices $M_1,\dots,M_d$ by
$$
(M_i)_{j,k}=\cos\Bigl(\frac{x_{j,i}-x_{k,i}}{2}\Bigr),\quad i=1,\dots,d,\quad j,k=1,\dots,n.
$$
By Bochner's theorem, the matrices $M_i$ are all positive semidefinite and, by the Schur Product Theorem \cite[Theorem 7.5.3]{HornJohnson},
also their Hadamard product $M=M_1\circ\dots\circ M_d$ is positive semidefinite. Obviously, $M$ has all its diagonal elements equal to one.
Finally, 
Corollary \ref{cor:2} shows that the matrix $M\circ M-\frac{1}{n}{E_n}$ with entries
$$
(M\circ M)_{j,k}-\frac{1}{n}=\prod_{i=1}^d \cos^2\Bigl(\frac{x_{j,i}-x_{k,i}}{2}\Bigr)-\frac{1}{n}
=\prod_{i=1}^d\frac{1+\cos(x_{j,i}-x_{k,i})}{2}-\frac{1}{n}
$$
is also positive semidefinite.

Hence, the integration problem on $F_d$ is intractable even when we allow negative weights $c_j$ in the quadrature formula \eqref{eq:1}.

{\bf Acknowledgement:} The author thanks Dmitriy Bilyk (University of Minnesota), Aicke Hinrichs (JKU Linz) and Erich Novak (FSU Jena) for fruitful discussions
and the anonymous referee, who proposed Theorems \ref{thm:ref1} and \ref{thm:ref2} as generalizations of Theorem \ref{thm:main1}.

\end{document}